\documentclass[11pt, oneside]{article}   	
\usepackage{geometry}                		
\usepackage{graphicx}				
\usepackage{amssymb}
\usepackage{amsthm}
\usepackage{color} 
\usepackage{enumerate}

\usepackage{amsmath}
\numberwithin{equation}{section}
\usepackage{cite}
\usepackage{booktabs,makecell,multirow}

\usepackage{pgfplots}
\usepackage{amsfonts}
\usepackage{amssymb}
\usepackage{indentfirst}
\usepackage{mathrsfs}
\usepackage{amsfonts,amsmath,latexsym,bm}
\usepackage{amsthm}
\usepackage{float}
\usepackage{amsmath}
\usepackage{booktabs,makecell,multirow}
\usepackage{url}

\usepackage{stmaryrd}
\usepackage{graphicx, epstopdf, epsfig}
\usepackage{makecell,multirow}
\usepackage{subfigure}
\usepackage{caption}
\usepackage[linesnumbered,boxed,ruled,commentsnumbered]{algorithm2e}

\usepackage[pdftex, pdfpagelabels, bookmarks, hyperindex, hyperfigures,
colorlinks, pagebackref]{hyperref}

\newcommand{\norm}[1]{\left\|#1\right\|}
\newtheorem{lem}{Lemma}[section]
\newtheorem{thm}{Theorem}[section]

\newtheorem{rem}{Remark}[section]

\newtheorem{assum}{Assumption}[section]

\begin{document}
\title{High order mixed finite elements with mass lumping for elasticity on triangular grids  \thanks
  {
    This work was supported in part by   National Natural Science Foundation of
    China (11771312).
  }
}
\author{Yan Yang$^{1}$\thanks{{Email:yyan2011@163.com}}, \
	\ Xiaoping Xie$^2$ \thanks{Corresponding author. Email: xpxie@scu.edu.cn}\\
	\\[2mm]
	\footnotesize{1  School of Sciences, Southwest Petroleum University, Chengdu 610500, China} 
	\\
	\footnotesize{ 2  School of Mathematics, Sichuan University, Chengdu 610064, China}}


%
\date{}							

\maketitle

\begin{abstract}
	A family of conforming mixed finite elements    with mass lumping on triangular grids are presented for linear elasticity.  The stress field is approximated by symmetric $H({\rm div})-P_k (k\geq 3)$  polynomial tensors enriched with higher order bubbles so as  to allow mass lumping, which can be viewed as the Hu-Zhang elements enriched with higher order interior bubble functions.  
	The displacement field is approximated by     $C^{-1}-P_{k-1}$ polynomial vectors enriched  with higher order terms to ensure the stability condition.  For both the proposed mixed elements and their mass lumping schemes, optimal error estimates are  derived for the stress with $H(\rm div)$ norm and the displacement with $L^2$ norm.  Numerical results confirm the theoretical analysis.

\end{abstract}



\section{Introduction}

Let $\Omega \subset \mathbb{R}^2$ be a polygonal region with   boundary $\partial\Omega$. 
We consider the   following   mixed variational system of linear elasticity based on the Helligner-Reissner principle: Find $({\sigma},{u})\in \Sigma\times V :=H({\rm div}, \Omega; \mathbb{S})\times L^2(\Omega; \mathbb{R}^2)$, such that
\begin{equation}\label{continuous}
\left\{ {\begin{array}{*{20}{c}}
  {(\mathcal{A}{\sigma} ,\tau ) + ({\rm div}\tau ,u ) = 0}&{\forall\tau  \in \Sigma ,} \\
  {-({\rm div}{\sigma} ,v) = (f,v)}&{\forall v \in V.}
\end{array}} \right.
\end{equation}
Here  ${\sigma}:\Omega\rightarrow  \mathbb{S}:=\mathbb{R}^{2\times 2}_{\text{sym}}$ denotes  the symmetric $2\times 2$ stress tensor field,   ${u}:\Omega\to \mathbb{R}^2$  the displacement field,  
and $\mathcal{A}{\sigma} \in \mathbb{S}$ the compliance tensor with
\begin{eqnarray}
\mathcal{A}{\sigma}:=\frac{1}{2\mu}\left({\sigma}-\frac{\lambda}{2\mu+2\lambda}\text{{\rm tr}}({\sigma})I\right),
\end{eqnarray}
where $\lambda> 0,\mu>0$ are the Lam\'e coefficients, $\text{{\rm tr}}({\sigma})$   the trace of ${\sigma}$,  $I$   the $2\times 2$ identity matrix, and
${f}$   the  body force.  $H({\rm div}, \Omega; \mathbb{S})$ denotes the space of square-integrable symmetric matrix fields with square-integrable divergence, and $L^2(\Omega; \mathbb{R}^2)$  the space of square-integrable vector fields. The $L^2$ inner products on vector and matrix fields are given by
\[\begin{gathered}
( v,w) := \int_\Omega  {v \cdot w dx}  = \int_\Omega  {\sum\limits_{i = 1}^2 {{v_i}{w_i}} dx} ,\;\;v=(v_1,v_2),  w=(w_1,w_2)\in V, \hfill \\
(\sigma ,\tau ) := \int_\Omega  { \sigma :\tau dx}  = \int_\Omega  {\sum\limits_{1 \leq i,j \leq 2} {{\sigma _{ij}}{\tau _{ij}}dx} } ,\;\; \sigma=(\sigma_{ij}) ,\tau=(\tau_{ij}) \in \Sigma, \hfill \\
\end{gathered} \]
respectively.
 
According to the standard theory of mixed  methods  \cite{Brezzi1991Mixed}, a mixed finite element discretization of the weak problem
 \eqref{continuous} requires the pair of  stress and displacement approximations to satisfy  two stability
conditions, i.e. a coercivity condition and an inf-sup condition.  These stability constraints   make it challengeable to construct    stable finite element pairs with   symmetric   stresses. In this field, we refer to \cite{Arnold1984A, Arnold2005Rectangular,Arnold2002Mixed,Awanou2012Two, Adams2005A, Hu2014Simple, Chen2011Conforming,Arnold2008Finite}   for some conforming mixed methods and to \cite{arnold2003nonconforming, Gopalakrishnan2011Symmetric, Hu2007Lower, Yi2006, Hu2018Non} 
   for some nonconforming methods. 
 In \cite{Hu2015Family,Hu2015tetrahedral}  Hu and Zhang designed a family of conforming symmetric mixed finite elements with optimal convergence orders for linear elasticity on  triangular and tetrahedral grids. Later Hu \cite{Hu2014Finite} extended the elements to   
    simplicial grids in $\mathbb{R}^n$ for any positive integer $n$. In these  elements,  the stress is approximated by symmetric $H(\text{div},\Omega;\mathbb{S})-P_k$ polynomial tensors and the displacement is approximated by $L^2(\Omega;\mathbb{R}^n)-P_{k-1}$ polynomial vectors for $k\ge n+1$.

However, for a  mixed finite element discretization based on \eqref{continuous}, a computational drawback is the need to solve an algebraic system of saddle point type like
\begin{equation}\label{mixed-dis}
\left( {\begin{array}{*{20}{c}}
	\mathbb{A} &\mathbb{B}^T\\
	-\mathbb{B} & \mathbb{O}
	\end{array}} \right)
\left( {\begin{array}{{c}}
	{X_1}\\
	{X_2}
	\end{array}} \right)=
\left( {\begin{array}{{c}}
	O\\
	{F}
	\end{array}} \right),
\end{equation}
where $\mathbb{A}$ is a  symmetric and positive definite (SPD) matrix corresponding to the term $(\mathcal{A}{\sigma} , \tau )$ in \eqref{continuous}, and $X_1$ and $X _2$ are the vectors of unknowns for the discrete stress and displacement approximations, respectively. 
One possible approach to resolve this difficulty  is  to apply  `mass lumping' on $(\mathcal{A}{\sigma} , \tau )$ so as to get a diagonal or block-diagonal matrix approximation, $\tilde{\mathbb{A}}$, of  the `mass matrix' $\mathbb{A}$. Replacing $\mathbb{A}$ with $\tilde{\mathbb{A}}$ in the discrete system \eqref{mixed-dis}, we  obtain 
$$X_1=-\tilde{\mathbb{A}}^{-1} \mathbb{B}^TX_2$$
and then
\begin{equation}\label{schur}
\mathbb{B}\tilde{\mathbb{A}}^{-1}\mathbb{B}^TX_2=F.
\end{equation}
 Notice that $\tilde{\mathbb{A}}$  is  diagonal or block-diagonal,  so   is $\tilde{\mathbb{A}}^{-1}$. This means that  the Schur complement   $-\mathbb{B}\tilde{\mathbb{A}}^{-1}\mathbb{B}^T$ is SPD. As a result,   by mass lumping the saddle point system \eqref{mixed-dis} is reduced to the SPD system \eqref{schur},   which can be solved efficiently by many fast algorithms.

The key to achieve mass lumping is to select appropriate numerical quadrature rule,  in which   the  quadrature nodes are required to match the finite element basis functions as well as  maintain sufficient numerical  integration accuracy. 
It has been shown that   mass lumping schemes can be constructed for some finite elements \cite{Ciarlet1978The,Baker1976The,Fix1972EFFECTS, Cohen2001Higher,Chin1999Higher,Mulder2013New,W2001HIGHER,Giraldo2006A,Liu2017Higher,Becache2002New,Cohen2005Mixed,Younes2006A}.
In \cite{Ciarlet1978The,Baker1976The,Fix1972EFFECTS}  the standard linear  triangular/tetrahedral elements with mass lumping were analyzed, where the   quadrature nodes are the vertices of the elements.  Traditional  higher order elements are not suitable for  mass lumping due to the requirements of numerical accuracy and stability, and  one has to use    finite element spaces  enriched with some bubble functions to adapt   mass lumping \cite{Cohen2001Higher,Chin1999Higher,Mulder2013New,W2001HIGHER,Giraldo2006A,Liu2017Higher}.   We note that a family of mixed rectangular and cubic finite elements with mass lumping  were constructed in \cite{Becache2002New}  for   linear elastodynamic problems, where the stress and displacement are approximated by symmetric $H({\rm div})-Q_k$ polynomial tensors and $L^2-Q_{k-1}$ polynomial vectors, respectively, and  the locations of the degrees of freedom for the    finite element spaces correspond to tensor products of one-dimensional quadrature
nodes associated with Gauss-Lobatto  (for stress) or  Gauss-Legendre (for velocity) quadrature formulas. 

 In this paper,  we first modify    Hu-Zhang's mixed conforming finite elements  \cite{Hu2015Family} to obtain a  family of new  elements which allow mass lumping. The stress field is approximated by symmetric $H({\rm div})-P_k (k\geq 3)$  polynomial tensors enriched with higher order bubbles, and  the displacement field   by     $C^{-1}-P_{k-1}$ polynomial vectors enriched  with higher order terms.  Error analysis is carried out for the new elements as well as their mass lumping schemes.

The remainder of this paper is organized as follows.   Section 2 introduces some preliminary results of mixed finite elements, including Hu-Zhang's elements. Sections 3 and 4  are devoted to the construction and  analysis of   the new mixed elements and their mass lumping schemes, respectively.  Finally,   Section 5 gives some numerical experiments   to verify the theoretical results.

\section {Preliminaries}

\subsection{Notations}
For integer $m\geq 0$, let  $H^m(\Omega; X)$ 
be the    Sobolev spaces  consisting of functions with
domain $\Omega$, taking values in    $X=\mathbb{S}$ or $\mathbb{R}^2$, and
with all derivatives of order at most $m$ square-integrable.  The norm and semi-norm on $H^m(\Omega; X)$ are denoted respectively by 
  $\|\cdot \|_m$ and $|\cdot |_m$. In particular, $H^0(\Omega;X)=L^2(\Omega;X)$. 

Suppose   $\mathcal{T}_h=\bigcup\{K\}$  to be a conforming and shape-regular triangulation of the domain $\Omega$ consisting of   triangles. 
For any $K\in \mathcal{T}_h$,  let  $h_K$ denote its diameter, and set $h: = \mathop {\max }\limits_{K \in {\mathcal{T}_h}}  {{h_K}} .$
We  use $P_m(K;X)$ to denote  the set of all polynomials  on  $K$ with  degree at most $m$ and taking values in $X$.  


Throughout the paper, we use $a\lesssim b$ $(a\gtrsim b)$ to denote $a\leq Cb$ $(a\geq Cb)$, 
where   $C$ is a generic positive constant   independent of mesh parameters $h$.

\subsection{Mixed finite element discretization}
%
%

Let $\Sigma_h\subset \Sigma, V_h\subset V$ be two finite-dimensional  spaces for the stress and displacement approximations, respectively. Then the mixed finite element discretization of (\ref{continuous}) reads: Find $(\sigma_h, u_h)\in\Sigma_{h}\times V_{h}$ such that
\begin{equation}\label{discrete1}
\left\{ {\begin{array}{*{20}{rl}}
	{({\cal A}{\sigma _h},\tau_h ) + ({\rm div}\tau_h ,{u_h}) = 0}&{\forall\tau_h  \in {\Sigma _{h}},}\\
	{ - ({\rm div}{\sigma _h},v_h) = (f,v_h)}&{\forall v_h \in {V_{h}}.}
	\end{array}} \right.
\end{equation}

According to the standard theory of mixed finite element methods \cite{Brezzi1974Existence,Brezzi1991Mixed},   the pair   of finite element spaces $\Sigma_h$ and $V_h$  needs  to satisfy the following stability conditions:
\begin{itemize}
\item K-ellipticity condition 
\begin{equation}\label{A1}
(\mathcal{A}\tau_h,\tau_h)\gtrsim \|\tau_h\|_{H({\rm div})}^2 \quad \forall \tau_h \in Z_h:= \left\{ {\left. {{\tau _h} \in {\Sigma _h}} \right|\;({{\rm div}}{\tau _h},v) = 0,\;\;\forall v \in {V_h}} \right\},
\end{equation}
where $\|\cdot\|_{H(\text{div})}$ is the norm on the space $\Sigma$ defined  by
$$\|\tau\|_{H(\text{div})}^2:=\|\tau\|_0^2+\|\text{div}\tau\|_0^2\quad \forall \tau\in \Sigma.$$ 

\item Discrete BB ( inf-sup ) condition 
\begin{equation}\label{A2}
  \mathop {\sup }\limits_{\tau_h  \in {\Sigma _h}} \frac{{({{\rm div}}\tau_h ,v_h)}}{{{\|\tau_h\| _{H({\rm div})}}}} \gtrsim {\|v_h\|_0}\quad \forall v_h \in {V_h}.
\end{equation}
\end{itemize}

%
%

\subsection{Hu-Zhang's mixed conforming elements}

For each $K\in\mathcal{T}_h$,  define an $H({\rm div})$ bubble function space,  $B_{k,K}$, of polynomials of degree $k$ by
\[{B_{k,K}}: = \left\{ {\tau  \in {P_k}(K;\mathbb{S}):{{\left. {\tau \nu } \right|}_{\partial K}} = 0} \right\},\]
where $\nu $ is the normal vector along $\partial K$. Introduce 
the local rigid motion space 
\begin{equation}\label{2-8}
R(K): = \left\{ {v \in {H^1}(K;{\mathbb{R}^2}): {{\nabla v + (\nabla {v})^T}} = 0} \right\}=\text{span}\left\{
\left( \begin{gathered}
1 \hfill \\
0 \hfill \\
\end{gathered}  \right),\ \left( \begin{gathered}
0 \hfill \\
1 \hfill \\
\end{gathered}  \right),\ \left( \begin{gathered}
- x_2 \hfill \\
x_1 \hfill \\
\end{gathered}  \right)\right\}
\end{equation}
 and its orthogonal complement space   with respect to $P_{k-1}(K;\mathbb{R}^2) (k\geq 2)$,  
\begin{equation}\label{2.9}
{R_k^\bot}(K) = \left\{ {v \in {P_{k - 1}}(K;{\mathbb{R}^2}): {{(v,w)}_K} = 0,\;\forall w \in R(K)} \right\}.
\end{equation}
The    following result holds.

\begin{lem}\label{lem2.1}\cite{Hu2014Finite}
	For any $K\in\mathcal{T}_h$ and $k\geq 2$, it holds that
	\begin{equation}
	R_k^\bot(K)={\rm div} B_{k,K}.
	\end{equation}
\end{lem}


For $k\geq 3$, introduce the following global finite element spaces  \cite{Hu2015Family}:
\begin{alignat}{1}  
&{\Sigma _{k,h}}: = \widetilde\Sigma _{k,h}+B_{k,h},\\
&{V_{k ,h}}: = \left\{ {v \in {L^2}(\Omega ;{\mathbb{R}^2}):\;{{\left. v \right|}_K} \in {P_{k - 1}}(K;{\mathbb{R}^2}),\;\forall K \in {\mathcal{T}_h}} \right\}, 
\end{alignat}
where
\begin{alignat}{1}  
& B_{k,h}:={\left\{ {\tau  \in H({\rm div},\Omega ;\mathbb{S}):{{\left. \tau  \right|}_K} \in {B_{k,K}},\;\forall K \in {{\cal T}_h}} \right\},}\\
&\widetilde\Sigma _{k,h}: = \left\{ {\tau  \in {H^1}(\Omega ;\mathbb{S}):\;{{\left. \tau  \right|}_K} \in {P_k}(K;\mathbb{S})},\ \forall K\in\mathcal{T}_h \right\}.
\end{alignat}
It is easy to see that   $\mathbb{S}$ has a canonical basis: 
$$\mathbb{T}_1:=\left(\begin{array}{cc}
1 & 0\\ 0 & 0
\end{array}\right),\
\mathbb{T}_2:=\left(\begin{array}{cc}
0 & 1\\ 1 & 0
\end{array}\right),\
\mathbb{T}_3:=\left(\begin{array}{cc}
0 & 0\\ 0 & 1
\end{array}\right).$$

For any edge $E$ of $\mathcal{T}_h$, let $t_E$ and $\nu_E$ be its unit tangent and norm vectors, respectively. Define
\begin{equation}\label{2-14}
T_E:=t_Et_E^T.
\end{equation}
Let $T_{E,j}^\bot\in\mathbb{S}$ ($j=1,2$) be two  orthogonal complement matrices  of $T_E$ with
\begin{equation}
T_{E,j}^\bot:T_E=0, \ T_{E,j}^\bot:T_{E,j}^\bot=1\ 
\text{ and } T_{E,1}^\bot:T_{E,2}^\bot=0.
\end{equation}
Here $A:B=\sum\limits_{i=1}^{n}\sum\limits_{i=1}^{n}a_{ij}b_{ij}$ for  $A=(a_{ij})_{i,j=1}^n$ and  $B=(b_{ij})_{i,j=1}^n$.
It has been shown in \cite{Hu2015Family} that
$$\mathbb{S}=\text{span}\{\mathbb{T}_1,\mathbb{T}_2,\mathbb{T}_3\}=\text{span}\{T_E,T_{E,1}^\bot,T_{E,2}^\bot\}.$$

Let $\left\{\chi_i\right\}_{i=1}^l$ be the set  of nodes for the Lagrange element of order $k$ and $\left\{\zeta_i\right\}_{i=1}^l$ be their associated Lagrange node basis functions such that
\begin{equation}\label{2-16}
\zeta_i(\chi_j)=\delta_{ij}, \ i,j=1, 2,\cdots,l.
\end{equation}
Then the basis functions of $\Sigma_{k,h}$ on $K$ fall into the following four classes  \cite{Hu2015Family}:
\begin{enumerate}[(1)]
	\item Vertex-based basis functions. If  $\chi_i$ is a vertex, the three associated basis functions of $\Sigma_{k,h}$ are $\zeta_i \mathbb{T}_j, j=1,2,3.$
	\item Volumed-based functions. If $\chi_i$ is a node inside $K$, the three associated basis functions of $\Sigma_{k,h}$ are $\zeta_i \mathbb{T}_j, j=1,2,3.$
	\item Edge-based basis functions with nonzero fluxes. If $\chi_i$ is a node on edge $E$ (not the vertex), the two associated basis functions of $\Sigma_{k,h}$ are $\zeta_i T_{E,j}^\bot, j=1,2.$
	\item Edge-based bubble functions. If $\chi_i\in K$ is a node on edge $E$ (not the vertex) shared by elements $K_1$ and $K_2$, then    $\left.\zeta_iT_E\nu_E\right|_E\equiv 0$ due to   \eqref{2-14}, and then the $H({\rm div})$ bubble functions in $\Sigma_{k,h}$ are $\left.\zeta_i\right|_{K_j}T_E, j=1,2.$
\end{enumerate}

\begin{thm}\cite{Hu2015Family,Hu2014Finite}
	Let $(\sigma,u)\in \Sigma\times V$ and $(\sigma_h,u_h)\in \Sigma_{h}\times V_{h}$, with $$\Sigma_{h}=\Sigma_{k,h} \ \text{ and }  \   V_h=V_{k,h},$$ solve (\ref{continuous}) and (\ref{discrete1}), respectively.  If $\sigma\in H^{k+1}(\Omega;\mathbb{S})$ and $v\in H^k(\Omega;\mathbb{R}^2)$ , then
	\begin{equation}
	{\left\| {\sigma  - {\sigma _h}} \right\|_{H({\rm div})}} + {\left\| {u - {u_h}} \right\|_0} \lesssim {h^k}\left( {{{\left\| \sigma  \right\|}_{k + 1}} + {{\left\| u \right\|}_k}} \right)
	\end{equation}
	and
	\begin{equation}
	{\left\| {\sigma  - {\sigma _h}} \right\|_0} \lesssim {h^{k + 1}}{\left\| \sigma  \right\|_{k + 1}}.
	\end{equation}
\end{thm}

\subsection{Mass lumping for Hu-Zhang elements?}

To solve the discrete system (\ref{discrete1}), we need to compute the inverse of the mass matrix corresponding to the term $(\mathcal{A}{\sigma_h} ,\tau_h )$. 

Let us first consider the local mass matrix on  element  $K\in\mathcal{T}_h$. Recall that $ \chi_i \ ( i=1,2,\cdots, l)$ are the nodes of the Lagrange element of order $k$ and $ \zeta_i\  ( i=1,2,\cdots, l)$ are the  associated Lagrange node basis functions.
 Thus, the  basis functions of $\Sigma_{k,h}$ on $K$ can be denoted by
$$\varphi_{3(m-1)+s}:=\zeta_m{T}_s, \ \ m=1,2,\cdots, l; s=1,2,3,$$
where $T_s\in \{\mathbb{T}_1,\mathbb{T}_2,\mathbb{T}_3\}$ if $\chi_m$ is a vertex or a node inside $K$, and $T_s\in \{T_E, T_{E,1}^\bot,T_{E,2}^\bot\}$ if $\chi_m$ is a node on edge $E$ (not the vertex). Then the local mass matrix $A_K$ on $K$ is given by
  $$(A_K)_{ij}:=\left.(\mathcal{A}\varphi_i,\varphi_j)\right|_K, \quad i,j=1,2,\cdots, 3l.$$

During the finite element method, we commonly evaluate the integrals approximately by using a numerical integration formula in each element $K$. To achieve mass lumping, the usual way is  to choose the quadrature points to be the nodes $\{\chi_i\}_{i=1}^l $ on $ K$, and   the quadrature rule is of the form
\begin{equation}\label{3-1}
\int_K fdx\approx I_{k,K}(f):=\sum_{i=1}^{l}w_if(\chi_i),
\end{equation}
where $\{w_i\}_{i=1}^l$ are the weights. 
Then we have
\begin{equation}\label{int-orth}
I_{k,K}(\zeta_i,\zeta_j)=w_i\delta_{ij}, \ \ i,j=1,2,\cdots, l.
\end{equation}
For $m,n=1,2,\cdots,l$ and $s,q=1,2,3$, set
$$i=3(m-1)+s,\ \ \ j=3(n-1)+q,$$
then from  \eqref{int-orth}
it follows
\begin{equation}\label{2-20}
\begin{gathered}
(A_K)_{ij} \approx (\tilde{A}_K)_{ij} : = {I_{k,K}}\left( {\mathcal{A}{\varphi _i},{\varphi _j}} \right) = {I_{k,K}}\left( {\mathcal{A}{\zeta_m}{T_{s}},{\zeta_n}{T_{q}}} \right) \hfill \\
\ \ \ \ \ = \left\{ {\begin{array}{*{20}{c}}
	{0,}&{m \ne n,} \\
	{w_m(\mathcal{A}T_s:T_q)}&{m = n.}
	\end{array}} \right. \hfill \\
\end{gathered}
\end{equation}
This means that the approximate local mass matrix $\tilde{A}_k$ is block-diagonal and of the form 
\begin{equation}\label{2-22}
\tilde{A}_K=diag(w_1B_1,w_2B_2,\cdots,w_lB_l),
\end{equation}
where $B_m(m=1,2,\cdots,l)$ are $3\times 3$ SPD matrices. For example, if $\chi_m$ is a vertex or a node inside $K$,  then
\[{B_m} = \frac{1}{{4\mu (\mu  + \lambda )}}\left( {\begin{array}{*{20}{c}}
	{2\mu  + \lambda }&0&{ - \lambda }\\
	0&{4(\mu  + \lambda )}&0\\
	{ - \lambda }&0&{2\mu  + \lambda }
	\end{array}} \right).\]

However,  the accuracy of  numerical  integration has to be   taken into account.  From the standard theory  \cite{Ciarlet1978The,Raviart1973The,Baker1976The}, the following condition is required to satisfy  so as  to maintain the accuracy of the   scheme  (\ref{discrete1}):
\begin{itemize}
\item[(A1)] The quadrature rule \eqref{3-1} must be exact for $P_{2k-2}$.
\end{itemize}
   Unfortunately,  the standard $P_k$ Lagrange elements   fail to satisfy this condition    for $k\geq3$ (cf. \cite{Cohen2001Higher}).
In other words, Hu-Zhang's elements do not allow mass lumping without loss of numerical accuracy.

\section{Modified mixed conforming finite elements for elasticity}
\subsection{${P}_{k,k'}$-Lagrange finite elements  for mass lumping}

As mentioned before, the standard $P_k$ Lagrange elements with $k\geq 3$  fail to satisfy  the accuracy condition, (A1), of the quadrature rule \eqref{3-1} for mass lumping.   For   wave problems,  as shown in  \cite{Cohen2001Higher,Chin1999Higher,Mulder2013New,W2001HIGHER,Giraldo2006A,Liu2017Higher},  an efficient way to address this difficulty is to construct a slightly larger finite element space  
\begin{equation}\label{p+}
{P}_{k,k'}(K;\mathbb{R}):=P_k(K;\mathbb{R})+bP_{k'-3}(K;\mathbb{R})=P_k(K;\mathbb{R}) \oplus b\sum\limits_{i = k-2}^{k' - 3} {P_{i}^{hom}(K;\mathbb{R})}.
\end{equation}
Here $k'>k$, and $b=\lambda_1\lambda_2\lambda_3$ is the bubble function on the element $K$ with $\lambda_i\ (i=1,2,3)$ being the barycentric coordinates. $P_i^{hom}(K;\mathbb{R})$ denotes the set of homogeneous polynomials on $K$ of degree   $i$.  The symbol ``$\oplus$" means that $P_k(K;\mathbb{R})  \bigcap b {P_{i}^{hom}(K;\mathbb{R})}=\{0\}$  for   $i=k-2,k-1, \cdots, k'-3$.

Let $\left\{\chi_i\right\}_{i=1}^r$ be the set of nodes  for the $P_{k,k'} $-Lagrange element. Then   the corresponding quadrature rule is of the form
\begin{equation}\label{4.2}
\int_K fdx \approx I_{k,k',K}(f):= \sum\limits_{i=1}^r w_if(\chi_i),
\end{equation}
where $\{w_i\}_{i=1}^r$ are the weights, and $\sum\limits_{i=1}^r w_i = \text{meas}(K)$.

To maintain the accuracy and stability of finite element scheme, the following two conditions are required 
(cf. \cite{Cohen2001Higher,Chin1999Higher}):
\begin{itemize}

\item[ (B1)]   The weights $w_i (i=1,\cdots,r)$ in     (\ref{4.2}) should be strictly positive;

\item[ (B2)]  The quadrature rule (\ref{4.2})  must be exact for $P_{k+k'-2}$.
\end{itemize}

  Table 1 lists several $P_{k,k'}$-finite elements  which satisfy (B1) and (B2)  with 
  $3\leq k\le 5$.
In the table, a given node $(\alpha_1,\alpha_2,\alpha_3)$ represents an  equivalence class which  includes all the nodes obtained by  taking all the permutations of the barycentric coordinates $\alpha_i$.  For instance, the class $(0,0,1)$ includes three points, $(0,0,1)$, $(0,1,0)$, and $(1,0,0)$; the  class $(\alpha,0,1-\alpha)$ includes 
\begin{equation}
(\alpha,0,1-\alpha), (0,\alpha,1-\alpha), (\alpha,1-\alpha,0), (1-\alpha,\alpha,0), (1-\alpha,0,\alpha), (0,1-\alpha,\alpha).
\end{equation}
%
%

\begin{table}
	\centering
	\caption{${P}_{k,k'}$-Lagrange  triangular elements.
	}
	\begin{tabular}{ @{} c|cc|cc|c @{} }
		\hline
		space &$k$&$k'$&class&weight&position parameters\\
		\hline
		$P_{3,4}$\cite{Cohen2001Higher,Chin1999Higher}&3&4&(0,0,1)&$(8-\sqrt 7)/720$&\\
		&&&$(\alpha,0,1-\alpha)$&$(7+4\sqrt{7})/720$&$1/2-\sqrt{1/(3\sqrt{7})-1/12}$\\
		&&&$(\alpha,\alpha,1-2\alpha)$&$7(14-\sqrt{7})/720$&$(7-\sqrt{7})/21$\\
		\hline
		$P_{3,5}$\cite{Chin1999Higher}&3&5&(0,0,1)&0.00356517965360224101681201&\\
		&&&$(\alpha,0,1-\alpha)$&0.0147847080884026469663777&0.307745941625991646104616\\
		&&&$(\alpha,\alpha,1-2\alpha)$&0.0509423265134759070757019&0.118613686396592868190663\\
		&&&$(\alpha,\alpha,1-2\alpha)$&0.0825897443227832246413973&0.425340125989747152025431\\
		\hline	
		$P_{4,5}$\cite{Chin1999Higher,Mulder2013New}&4&5&(0,0,1)&1/315&\\
		&&&$(1/2,0,1/2)$&4/315&\\
		&&&$(\alpha,0,1-\alpha)$&3/280&$1/2(1-1/\sqrt{3})$\\
		&&&$(\alpha,\alpha,1-2\alpha)$&$163/2520-47\sqrt{7}/8820$&$(5-\sqrt{7})/18$\\
		&&&$(\alpha,\alpha,1-2\alpha)$&$163/2520+47\sqrt{7}/8820$&$(5+\sqrt{7})/18$\\
		\hline	
		$P_{4,6}$\cite{Mulder2013New}&4&6&(0,0,1)&0.00150915593385883937469324&\\
		&&&$(1/2,0,1/2)$&0.0101871481261788846308014&\\
		&&&$(\alpha,0,1-\alpha)$&0.00699540146387514358396201&0.199632107119457219140683\\
		&&&$(1/3,1/3,1/3)$&0.0660095591593093891810431\\
		&&&$(\alpha,\alpha,1-2\alpha)$&0.0234436060814549086935898&0.0804959191700374444460458\\
		&&&$(\alpha,\beta,1-\alpha-\beta)$&0.0477663836054936418696553&0.107591821784867520262175,\\
		&&&&&0.302912783038363411733216\\
		\hline	
		$P_{5,7}$\cite{Chin1999Higher,Mulder2013New}&5&7&(0,0,1)&0.000709423970679245979296007\\
		&&&$(\alpha,0,1-\alpha)$&0.00348057864048921065844268&0.132264581632713985353888\\
		&&&$(\alpha,0,1-\alpha)$&0.00619056500367662911411813&0.363298074153686045705506\\
		&&&$(\alpha,\alpha,1-2\alpha)$&0.0116261354596175711394984&0.0575276844114101056608175\\
		&&&$(\alpha,\alpha,1-2\alpha)$&0.0459012376307628573770191&0.256859107261959076063891\\
		&&&$(\alpha,\alpha,1-2\alpha)$&0.0345304303772827935283885&0.457836838079161101938503\\
		&&&$(\alpha,\beta,1-\alpha-\beta)$&0.0272785759699962595486715&0.0781925836255170219988860,\\
		&&&&&0.221001218759890007978128\\
		\hline					
	\end{tabular}
\end{table}

\subsection{Modified mixed     element spaces for elasticity}

Inspired by the ${P}_{k,k'}$-Lagrange   elements which allow  mass lumping, in this subsection
 we shall construct a family of new mixed conforming element spaces   based on the modification of Hu-Zhang's elements. 

For $k'>k\geq 3$,
set 
\begin{equation}
\Lambda_{k,k'}:=\left\{\tau\in H({\rm div},\Omega;\mathbb{S}): \left.\tau\right|_K=\sum\limits_{i = k-2}^{k'-3} b{P_{i}^{hom}(K;\mathbb{S})},\ \forall K\in\mathcal{T}_h \right\}.
\end{equation}
Then the modified global finite element spaces for the stress and displacement   are given by
\begin{align}
\Sigma_{k,k',h}:=& \Sigma_{k,h} \oplus \Lambda_{k,k'},\label{3-5.06}\\
{V_{k,k',h}}: =& V_{k,h}+{\rm div}\Lambda_{k,k'}.\label{3-16}
\end{align}
Obviously we have ${\rm div}\Sigma_{k,k',h} \subset {V_{k,k',h}}$.
\begin{rem}
If we define
\begin{align}
{B_{k,k',h}}: =&\left\{ {\tau  \in {{P}_{k,k'}}(K;\mathbb{S}):{{\left. {\tau \nu } \right|}_{\partial K}} = 0,\ \forall K\in\mathcal{T}_h}\right\}=B_{k,h} + \Lambda_{k,k'} ,\label{3-2}\\
\widetilde{\Sigma}_{k,k',h}:= &\left\{\tau\in H^1(\Omega;\mathbb{S}):\left.\tau\right|_K\in P_{k,k'}(K;\mathbb{S}),\forall K\in \mathcal{T}_h \right\}=\widetilde{\Sigma}_{k,h}+ \Lambda_{k,k'},
\end{align}
then we can also write
\begin{align}
\Sigma_{k,k',h}=&\widetilde{\Sigma}_{k,k',h} +{B_{k,k',h}}.
\end{align}
\end{rem}


Let $\left\{\chi_i\right\}$ be set of the nodes for the $P_{k,k'}$-Lagrange element, and $\left\{\zeta_i\right\}$ be the corresponding nodal basis functions satisfying
\begin{equation}\label{3-5}
\zeta_{i}(\chi_j)=\delta_{ij}.
\end{equation}
Similarly to  Hu-Zhang's elements described in Section 2.3,  for each node  $\chi_i$ the associated basis functions of $\Sigma_{k,k',h}$ on $K$  
are given as follows: 
\begin{enumerate}[(1)]
	\item  $\zeta_i \mathbb{T}_j  \ ( j=1,2,3)$, if  $\chi_i$ is a vertex or a  node inside $K$; 
	\item $\zeta_i T_{E,j}^\bot \ (j=1,2) $ and $\left.\zeta_i\right|_{K_j}T_E \ (j=1,2)$, if $\chi_i$ is a node on edge $E$ (not the vertex) shared by elements $K_1$ and $K_2$.
\end{enumerate}

	\subsection{Stability results}
This subsection  is devoted to the stability  analysis and error estimation of the mixed finite element scheme (\ref{discrete1}) with 
$$\Sigma_{h}=\Sigma_{k,k',h} \ \text{ and }  \   V_h=V_{k,k',h}.$$

 Let $\hat{K}$ be the reference element with    vertexes $(0,0), (0,1), (1,0)$. For each $K\in\mathcal{T}_h$, let $F_K$ denote the affine map from $\hat{K}$ onto $K$ so that $F_K(\hat{K})=K$. Let $\chi_0, \chi_1, \chi_2$ be the vertices of triangle $K\in \mathcal{T}_h.$ The referencing mapping is then of the  form
\[x = {F_K}(\hat x ) = {\chi _0} + \left( {\begin{array}{*{20}{c}}
	{\chi_1 - {\chi _0}}& {{\chi _2} - {\chi _0}}
	\end{array}} \right)\hat x:= \chi_0+B_K\hat{x}, \ \ \forall \hat x\in \hat K. \]
By the shape regularity    of $\mathcal{T}_h$, it holds  that
\begin{equation}\label{4-1}
\left\|B_K\right\|_0\lesssim h,\ \ \left\|B_K^{-1}\right\|_0\lesssim h^{-1}.
\end{equation}
We  need to introduce the Piola transform   as follows. Given $\hat{\tau} : \hat{K}\mapsto \mathbb{S}$, $\tau: K \mapsto \mathbb{S}$ is defined by
\begin{equation}\label{Piola}
\tau(x):= B_K\hat{\tau}(\hat{x})B^T_K.
\end{equation}
Clearly this sets up a one-to-one correspondence between  $ L^2(\hat{K};\mathbb{S})$ and $L^2(K;\mathbb{S})$ with 
\begin{equation}
{{{\rm div}\tau (x)} } = {B_K}\widehat{\rm div}\hat \tau (\hat x).
\end{equation}

		Standard scaling arguments yield  the following lemma.
		\begin{lem}
			For any $  K\in\mathcal{T}_h$ and  $\hat{\tau}\in {\hat P}_{k,k'}(\hat{K};\mathbb{S})$,    let $\tau$ is given  by (\ref{Piola}). Then for $1\leq q\leq k$,   
			\begin{equation}\label{3-11}
			{\left| \tau  \right|_{q,K}} \lesssim h^{2-q} {\left| {\det {B_K}} \right|^{\frac{1}{2}}} \  {\left| {\hat \tau } \right|_{q,\hat K}},
			\end{equation}
			\begin{equation}\label{3-12}
			{\left| \hat\tau  \right|_{q,\hat{K}}} \lesssim h^{q-2} {\left| {\det {B_K}} \right|^{-\frac{1}{2}}} \  {\left| { \tau } \right|_{q, K}}.
			\end{equation}
		\end{lem}

		\begin{assum}\label{assum4.1}
			For any $\tau\in\Sigma_{k,k',h}\subset H({\rm div},\Omega;\mathbb{S})$, if
			\begin{equation*}
			{\rm div}\tau|_K=0\ \  \ \ \forall K\in\mathcal{T}_h,
			\end{equation*}
			then
			$\tau=0.$
		\end{assum}
		Define the piecewise $m$-order semi-norm $|\cdot |_{m,h} \ $  $(1\leq m\leq k')$ on ${\Sigma}_{k,k',h}$ as follows:
		
	\begin{equation}
	|\tau_h|_{m,h}:= \left(\sum_{K\in\mathcal{T}_h}|\tau_h|_{m,K}^2\right)^{\frac{1}{2}},\ \tau_h\in {\Sigma}_{k,k',h}.
	\end{equation}
	
		\begin{lem}\label{lem3.3}
			Suppose that  $\tau\in\Sigma_{k,k',h}$  satisfies   Assumption \ref{assum4.1}, then
			\begin{equation}\label{4.5}
			{\left\| \tau  \right\|_0} \lesssim h{\left| \tau  \right|_{1,h}} \lesssim h{\left\| {{{\rm div}}\tau } \right\|_0}.
			\end{equation}

		\end{lem}		
		\begin{proof}
			For any $K\in\mathcal{T}_h$, let $\hat{K}$ be the reference element. By \eqref{Piola}, we have
			$$\hat\tau(\hat{x})=B_K^{-1}\tau(x)(B_K^{-1})^{T},$$
			and $\hat{\tau}(\hat{x})$ satisfies   Assumption \ref{assum4.1}. Thus, both $\norm{\widehat{\rm div}\hat{\tau}}_{0,\hat{K}}$ and $\left|\hat{\tau}\right|_{1,\hat{K}}$ are    norms on $\hat{K}$. Then it holds 
			\begin{equation*}\label{3.16}
			\left\|\hat{\tau}\right\|_{0,\hat{K}}\lesssim \left|\hat{\tau}\right|_{1,\hat{K}}\lesssim \norm{\widehat{\rm div}\hat{\tau}}_{0,\hat{K}},
			\end{equation*}
			which, together with  \eqref{3-11} and \eqref{3-12},  implies 
			\begin{equation*}\label{}
			\begin{array}{rlr}
			{\left\|\tau\right\|_{0,K}} =&\norm{B_K\hat\tau B_K^T}_{0,K}\\
			\leq &\norm{B_K}\  {\left| {\det {B_K}} \right|^{\frac{1}{2}}} \norm{\hat\tau}_{0,\hat{K}}\ \norm{B_K^T}\\
			\lesssim &  h^2\left| \det B_K \right|^{\frac{1}{2}}\left|\hat \tau\right|_{1,\hat{K}} & \\
			\lesssim &h\left|\tau\right|_{1,K} & 
			\end{array}
			\end{equation*}
			and
			\begin{equation*}\label{14-8}
			\begin{array}{rlr}
			{\left| \tau \right|_{1,K}} \lesssim &h {\left| {\det {B_K}}  \right|^{\frac{1}{2}}} \  {\left| {\hat \tau} \right|_{1,\hat K}} & \\
			\lesssim &h {\left| {\det {B_K}} \right|^{\frac{1}{2}}} \  {\left\| {\widehat{\rm div}\hat \tau } \right\|_{0,\hat K}}= h   {\left| {\det {B_K}} \right|^{\frac{1}{2}}}   {\left\| B_K^{-1}{{\rm div}\tau } \right\|_{0,K}}& \\
			 \lesssim &{\left\| {{\rm div}\tau } \right\|_{0,K}}.
			\end{array}
			\end{equation*}		
			This completes the proof.		
 			\end{proof}

		In view of  the definitions in \eqref{3-2} and  \eqref{2-8}, 
	  integration by part yields 
	\begin{equation}
		\int_K{\rm div}\tau_h\cdot w_hdx =0\ \ \ \forall \tau_h\in B_{k,k',h},\ w_h\in R(K) , \  K\in\mathcal{T}_h.
	\end{equation}
	Analogous to \eqref{2.9}, we define 
	\begin{equation}\label{3-9}
	{R_{k,k'}^\bot}(K) := \left\{ {v \in V_{k,k',h}: {{(v,w)}_K} = 0,\;\forall w \in R(K)} \right\}.
	\end{equation}
	
	By following the same routines as in \cite{Hu2014Finite}, we can easily derive the following two lemmas.
	\begin{lem}\label{lem3.4}
		For any $K\in\mathcal{T}_h$ and $k\geq 2$,  it holds that
			\begin{equation}\label{3-10}
		{R}_{k,k'}^\bot(K)={\rm div}  \left. B_{k,k',h}\right|_K.
		\end{equation}
	\end{lem}
%
%
	
	\begin{lem}\label{lem3.5}
		For any $v_h\in V_{k,k',h}$, there exists a $\tau_h\in {\Sigma}_{k,k',h}$ such that  
		$$\int_K ({\rm div}\tau_h - v_h)\cdot pdx=0  \quad  \forall p\in R(K), K\in\mathcal{T}_h$$
		and $$ \|\tau_h\|_{H({\rm div})}\lesssim \|v_h\|_0.$$
	\end{lem}

We are now in a position to show the existence and uniqueness result.	
\begin{thm}\label{thm3.0}
The mixed finite element scheme (\ref{discrete1}) with 
$\Sigma_{h}=\Sigma_{k,k',h} \ \text{ and }  \   V_h=V_{k,k',h}$ admits a unique solution  $(\sigma_h, u_h)\in \Sigma_{ h}\times {V}_{h}$.
	\end{thm}

\begin{proof} It suffices to prove the K-ellipticity \eqref{A1} and the discrete BB inequality  \eqref{A2}. Note that 
	 \eqref{A1}  follows from the fact   ${\rm div}\Sigma_{k,k',h}\subset V_{k,k',h}$. 
	 
	 We follow a similar way in \cite{Hu2014Finite}   to show \eqref{A2}.
	 For any given $v_h\in V_{k,k',h}$, by Lemma \ref{lem3.5}, there exists  $\tau_1\in {\Sigma}_{k,k',h}$ with 
	$$\int_K ({\rm div}\tau_1-v_h)\cdot  pdx=0,\ \forall p\in R(K), \ K\in\mathcal{T}_h$$
	and $$ \|\tau_1\|_{H({\rm div})}\lesssim \|v_h\|_0.$$ 
	Then, by Lemma \ref{lem3.4} there exists  $\tau_2\in B_{k,k',h}$  with
	$${\rm div\tau_2}=v_h-{\rm div}\tau_1,\ \ \ \|\tau_2\|_0=\min\{\|\tau\|_0:\ {\rm div}\tau=v_h-{\rm div}\tau_1,\ \tau\in B_{k,k',h}\}.$$
	Thus, if ${\rm div}\tau_2=0$, then $\tau_2=0$, i.e.   $\tau_2$ satisfies  Assumption \ref{assum4.1}. Hence, by Lemma \ref{lem3.3} we have
	\begin{equation}
	\|\tau_2\|_{H({\rm div})}\lesssim \|v_h-{\rm div}\tau_1\|_0\lesssim  \|v_h\|_0.
	\end{equation}
	Finally, set $\tau_h:=\tau_1+\tau_2$, which implies that 
	\begin{equation}
	{\rm div}\tau_h=v_h \ \ \  {\text{and}}\ \ \|\tau_h\|_{H({\rm div})}\lesssim \|v_h\|_0.
	\end{equation}
	This means that the discrete BB inequality \eqref{A2}  holds.	
\end{proof}
	
	\begin{rem}\label{rem3.1}
		According to Lemma \ref{lem3.4} and \ref{lem3.5}, we can derive that there exists an interpolation $\Pi_h: H^1(\Omega;\mathbb{S})\mapsto \Sigma_{k,k',h}$ such that for any $\tau\in H^1(\Omega;\mathbb{S})$,
		$$({\rm div}(\tau-\Pi_h\tau),v_h)_K=0,\ \ \forall K\in\mathcal{T}_h,\ \forall {v_h} \in {V}_{k,k',h}. $$
		Furthermore, if $\tau\in H^{k+1}(\Omega;\mathbb{S})$, then
		\begin{equation}\label{4.11}
		{\left\| {\tau  - {\Pi _h}\tau } \right\|_0} \leq C{h^{k +1}}{\left\| \tau  \right\|_{k +1}}.
		\end{equation}
		This shows that the operator $\Pi_h: H^1(\Omega;\mathbb{S})\mapsto \Sigma_{k,k',h}$ has the following  commutative property:  
		\begin{equation}\label{commu}
		P_{h} {\rm div}\tau ={\rm div}\Pi_h\tau\quad \forall \tau\in H^1(\Omega;\mathbb{S}).
		\end{equation}
		Here $P_{h} : L^2(\Omega;\mathbb{R}^2)\mapsto V_{k,k',h}$ is the $L^2$ projection operator.
	\end{rem}

	By the stability conditions \eqref{A1}-\eqref{A2} and   Remark \ref{rem3.1},     we easily obtain the following error estimates.
	
\begin{thm}\label{thm3-1}
Let $(\sigma,u)\in \left(\Sigma\bigcap  H^{k+1}(\Omega;\mathbb{S})\right)\times \left(V\bigcap  H^k(\Omega;\mathbb{R}^2)\right)$ and $(\sigma_h,u_h)\in \Sigma_{h}\times V_{h}=\Sigma_{k,k',h}\times {V}_{k,k',h}$  solve (\ref{continuous}) and (\ref{discrete1}), respectively. Then it holds that
	\begin{equation}\label{3.10}
	{\left\| {\sigma  - {\sigma _h}} \right\|_{H({\rm{div}})}} + {\left\| {u - {u_h}} \right\|_0} \lesssim
	 {h^k}\left( {{{\left\| \sigma  \right\|}_{k + 1}} + {{\left\| u \right\|}_k}} \right)
	\end{equation}
	and
	\begin{equation}\label{3.11}
	{\left\| {\sigma  - {\sigma _h}} \right\|_0} \lesssim {h^{k +1 }}{\left\| \sigma  \right\|_{k+1}}.
	\end{equation}
\end{thm}

\section{Mass lumping mixed finite element method}

\subsection{Mass lumping scheme}
As   mentioned before, the mixed scheme (\ref{discrete1}) leads to an algebraic system of saddle point type. One approach to address this issue is applying mass lumping.  

The mass lumping scheme for  (\ref{discrete1}) is described as follows:  Find $(\sigma_h,u_h)\in {\Sigma}_{k,k',h} \times V_{k,k',h}$, such that
\begin{equation}\label{discrete2}
\left\{ {\begin{array}{*{20}{c}}
	{(\mathcal{A}\sigma_h ,\tau_h )_h + ({\rm div}\tau ,u_h) = 0}&{\forall\tau_h  \in {\Sigma}_{k,k',h} }, \\
	{-({\rm div}\sigma_h,v_h) = (f,v_h)}&{\forall v_h \in V_{k,k',h}.}
	\end{array}} \right.
\end{equation}
Here $ ( \mathcal{A}\sigma_h ,\tau_h)_h : = \sum\limits_{K\in\mathcal{T}_h}( \mathcal{A}\sigma_h ,\tau_h)_{h,K}$ with  
$$( \mathcal{A}\sigma_h ,\tau_h)_{h,K}:=I_{k,k',K}(\mathcal{A}\sigma_h :\tau_h),$$ and $I_{k,k',K}$ is the quadrature operator in   (\ref{4.2})   satisfing 
the conditions (B1) and (B2).

The following lemma  shows that  the quadrature rule (\ref{4.2}) produces a coercive bilinear form $( \cdot,\cdot)_h$.
\begin{lem}\label{lem4.2}
	It holds that
	\begin{equation}\label{coercivity-2}
	(\mathcal{A}\tau,\tau)_h\gtrsim \left\|\tau\right\|_0^2 \ \ \ \forall \tau \in {\Sigma}_{k,k',h}.
	\end{equation}
\end{lem}
\begin{proof}
	Recall that  $\left\{\chi_i\right\}_{i=1}^r$ are the nodes for the $P_{k,k'}$-Lagrange element, and $\left\{\zeta_i\right\}_{i=1}^r$ are the corresponding nodal basis functions satisfying \eqref{3-5}. Then, for any $\tau\in{\Sigma}_{k,k',h}$ we can denote $$\left.\tau\right|_K  = \sum\limits_{i = 1}^r\sum\limits_{j = 1}^3  {{c_{ij}}{\zeta _i}{T_j}},$$
    where $T_j\in \{\mathbb{T}_1,\mathbb{T}_2,\mathbb{T}_3\}$ if $\chi_m$ is a vertex or a node inside $K$, and $T_j\in \{T_E, T_{E,1}^\bot,T_{E,2}^\bot\}$ if $\chi_m$ is a node on edge $E$ (not the vertex). Thus,
\begin{eqnarray*}
 (\mathcal{A}\tau ,\tau )_{h,K} &=&  \left( \sum\limits_{i = 1}^r\left(\sum\limits_{j=1}^3 c_{ij}\mathcal{A}T_j\right)\zeta_i, \sum\limits_{s = 1}^r\left(\sum\limits_{t=1}^3 c_{st}T_t\right)\zeta_s\right)_{h,K} \\
 & =& \sum\limits_{i = 1}^r\sum\limits_{s = 1}^rI_{k,k',K}\left( (\zeta_i\sum\limits_{j=1}^3 c_{ij}\mathcal{A}T_j ): (\zeta_s\sum\limits_{t=1}^3 c_{st}T_t)\right)  \\
 &=& \sum\limits_{i = 1}^r w_i\left(\sum\limits_{j=1}^3 c_{ij}\mathcal{A}T_j : \sum\limits_{t=1}^3 c_{it}T_t\right)  \\
& \gtrsim& {\sum\limits_{i = 1}^r{w_i}\sum\limits_{j = 1}^3 {c_{ij}^2} }  
  \gtrsim  h^2\|\tau\|_{0,K}^2 ,
\end{eqnarray*}
where $w_i$ are the weights in (\ref{4.2}). As  a result, 
$$(\mathcal{A}\tau ,\tau )_{h} \gtrsim \sum\limits_{K\in\mathcal{T}_h} h^2\|\tau\|_{0,K}^2\gtrsim \left\|\tau\right\|_0^2,$$
which completes the proof.
\end{proof}

This coercivity  lemma, together with the discrete BB condition \eqref{A2}, yields the following conclusion.
\begin{lem}
	The mass lumping scheme   (\ref{discrete2}) admits  a unique solution.
\end{lem}

\subsection{Error estimation}


 
In light of the stability conditions \eqref{coercivity-2} and \eqref{A2} and standard techniques, we easily derive the following result.
\begin{lem}\label{lem4.3}
	Let $(\sigma,u)\in{\Sigma}\times V$ and $(\sigma_h,u_h)\in{\Sigma}_{k,k',h}\times V_{k,k',h}$ be the solutions of (\ref{continuous}) and  (\ref{discrete2}), respectively. Then  
	{\small \begin{equation}                                                                                                                                                                         
	\begin{array}{rl}
	{\left\| {\sigma  - {\sigma _h}} \right\|_{H({\rm{div}})}} + {\left\| {u - {u_h}} \right\|_0} \lesssim   \left\| u-P_hu\right\|_0+ \inf\limits_{\tilde{\tau}_h\in{\Sigma}_{k,k',h}} \left( {{{\left\| {\sigma  - \tilde{\tau}_h} \right\|}_{H({\rm div})}} + \mathop {\sup }\limits_{\tau_h  \in {\Sigma} _{k,k',h} } \frac{{ {{E_h}(\tilde{\tau}_h,\tau_h )} }}{{{{\left\| \tau_h  \right\|}_{H({\rm{div}})}}}}} \right),
	\end{array}
	\end{equation}	}
	where  $P_{h} : L^2(\Omega;\mathbb{R}^2)\mapsto V_{k,k',h}$ is the $L^2$ projection operator, and 
	 \begin{equation}\label{4-23}
E_h(\tilde{\tau}_h,\tau_h):=(\mathcal{A}\tilde{\tau}_h, \tau_h)- (\mathcal{A}\tilde{\tau}_h,\tau_h)_h = \sum\limits_{K\in\mathcal{T}_h}\left(\int_K\mathcal{A}\tilde{\tau}_h:\tau_h dx-I_{k,k',K} (\mathcal{A}\tilde{\tau}_h: \tau_h)  \right).
\end{equation}

 \end{lem}

Let ${ W}_h$ be a space 
 satisfying
$${\Sigma}_{k,h} \subseteq {W}_h\subseteq {\Sigma}_{k,k',h}$$
and consisting of piecewise polynomial tensors of  degree at most $\tilde{k}$,   $k\leq \tilde{k}\leq k'.$
Then we have the following estimate  for $E_h(\tilde{\tau}_h,{\tau}_h)$, 
which can be viewed as an
extended version of    \cite[Lemma 5.2]{Cohen2001Higher}.
\begin{lem}\label{lem4-2}
	If
	\begin{equation}\label{4-190}
	1\leq p\leq k-1+(k'-\tilde{k}),\ 0\leq q\leq k-1,
	\end{equation}
	then for all $(\tilde{\tau}_h,{\tau}_h)\in{\Sigma}_{k,k',h}\times { W}_h$, it holds
	\begin{equation}
	\left|E_h(\tilde{\tau}_h,{\tau}_h)\right|\lesssim h^{p+q}\left|\tilde{\tau}_h\right|_{p,h}\cdot \left|{\tau}_h\right|_{q,h}.
	\end{equation}
\end{lem}

\begin{proof}
	
	For any $K=F_K(\hat{K})\in \mathcal{T}_h$ with $x=F_K(\hat x)$, we set
	$$ \widehat{{\tilde{\tau}_h}}(\hat{{ x}}) :=\tilde{\tau}_h(x)|_K ,\ \ \ \ \widehat{\tau_h}(\hat x) = \tau_h(x) |_K.$$
	By    scaling arguments we have
	\begin{equation}\label{4-14}
	\left|\widehat{\tilde{\tau}_h}\right|_{p,\hat{K}}\lesssim h^p\left|\det B_K\right|^{-\frac{1}{2}}\left|\tilde{\tau}_h\right|_{p,K},\ \ \ \left|\widehat{\tau_h}\right|_{p,\hat{K}}\lesssim h^p\left|\det B_K\right|^{-\frac{1}{2}}\left| {\tau}_h\right|_{p,K}.
	\end{equation}
	Then
	 \begin{equation}\label{Eh}
	\left| {{E_h}(\tilde{\tau}_h ,{\tau}_h )} \right| = \sum\limits_{K \in {\mathcal{T}_h}}^{} {\left| {{E_{h,K}}(\tilde{\tau}_h ,{\tau}_h )} \right|} 
	= \sum\limits_{K \in {\mathcal{T}_h}} {\left| {\det {B_K}} \right|\hat E_{h,\hat{K}}\left( {\widehat{\tilde{\tau}_h} ,\widehat{\tau_h} } \right)}  .
	\end{equation}
	From (\ref{4-190}) it follows
	\begin{equation*}
	\begin{array}{l}
	0 \leq p - 1 + \tilde k \le k + k' - 2,\;\\
	0 \leq q - 1 + k' \le k + k' - 2,\;\\
	0\leq p - 1 + q - 1 \le k + k' - 2.
	\end{array}
	\end{equation*}
	Let $\hat{\Pi}_j$ denote the $L^2$ projection from $\hat L^2(\hat K;\mathbb{S})$ onto $\hat P_j(\hat{K};\mathbb{S})$. By (B2),  the quadrature rule (\ref{4.2})  is exact for $P_{k+k'-2}$. Thus,
	\begin{equation*}
	\begin{array}{l}
	\left| {\hat E_{h,\hat{K}}(\widehat{\tilde{\tau}_h} ,\widehat{\tau_h} )} \right| = \left| {\hat E_{h,\hat{K}}(\widehat{\tilde{\tau}_h}  - {\hat\Pi _{p - 1}}\widehat{\tilde{\tau}_h} ,\widehat{\tau_h}  - {\hat\Pi _{q - 1}}\widehat{\tau_h} )} \right|\\
	\lesssim{\left\| {\widehat{\tilde{\tau}_h}  - {\hat\Pi _{p - 1}}\widehat{\tilde{\tau}_h} } \right\|_{0,\hat K}} \cdot {\left\| {\widehat{\tau_h}  - {\hat\Pi _{q - 1}}\widehat{\tau_h} } \right\|_{0,\hat K}}\\
	\lesssim{\left| \widehat{\tilde{\tau}_h}  \right|_{p,\hat K}} \cdot {\left| \widehat{\tau_h}  \right|_{q,\hat K}} \lesssim{h^p}{\left| {\det {B_K}} \right|^{ - \frac{1}{2}}}{\left| \tilde{\tau}_h \right|_{p,K}} \cdot {h^q}{\left| {\det {B_K}} \right|^{ - \frac{1}{2}}} {\left| {\tau}_h  \right|_{q,K}}\ \ \ \ \ \ {\rm by}\ \eqref{4-14}\\
	\lesssim h^{p+q}\left|\det B_K\right|^{-1} |\tilde{\tau}_h|_{p,K}|\tau_h|_{q,K},
	\end{array}
	\end{equation*}
	which,together with \eqref{Eh},  yields the desired result.
\end{proof}

\begin{rem} If taking ${ W}_h={\Sigma} _{k,k',h}$ and $p=k-1, q=0$ in Lemma \ref{lem4-2}, then  we obtain
$$\left|E_h(\tilde{\tau}_h,{\tau}_h)\right|\lesssim h^{k-1}\left|\tilde{\tau}_h\right|_{k-1,h}\cdot \left|{\tau}_h\right|_{0,h}, \ \forall \tilde{\tau}_h,{\tau}_h\in {\Sigma} _{k,k',h}, $$
which yields 
\begin{equation}\label{subopt}
{\mathop {\sup }\limits_{{\tau _h} \in {{\Sigma} _{k,k',h}}} \frac{{  {{E_h}(\tilde{\tau}_h,{\tau _h})}  }}{{{{\left\| {{\tau _h}} \right\|}_{H({\rm{div}})}}}}}
 \lesssim{h^{k - 1}}{\left| \tilde{\tau}_h  \right|_{k - 1,h}}.
 \end{equation}
This inequality, together with Lemma \ref{lem4.3}, leads to an error estimate like
	\begin{equation}
	\begin{gathered}
	\left\|\sigma-\sigma_h\right\|_{H({\rm div})} + \left\|u-u_h\right\|_0\lesssim h^{k-1}\left(\left\|\sigma\right\|_k +\left\|u\right\|_{k-1}\right),
	\end{gathered}
	\end{equation}
provided that $\sigma\in H^k(\Omega,\mathbb{S})$ and $u\in H^{k-1}(\Omega,\mathbb{R}^2)$.
Note that	such an estimate is not optimal.  	
\end{rem}

In what follows we will apply a more elaborate analysis to get  a better estimate  for the consistency error 
 than \eqref{subopt}. 
To this end, we set, for any $K \in \mathcal{T}_h$,  
$${\Xi}_j:=b  P^{hom}_j(K; \mathbb{S}),\ \ k-2\leq j\leq k'-3.$$
Here we recall that $P_j^{hom}(K; \mathbb{S})$ denotes the set of homogeneous polynomial tensors of degree  $j$. 
On the reference element $\hat K$ with vertexes $(0,0), (1,0)$ and $(0,1)$,  the bubble function reads $\hat b=\hat{x}_1\hat{x}_2(1-\hat{x}_1- \hat{x}_2)$. Let $\left\{\hat\psi_i\right\}_{i=0}^{j}$ be the basis of the space $\hat{b}\hat{P}_{j}^{hom}(\hat{K};\mathbb R)$, then
\begin{equation*}\label{4-40}
\hat{\psi}_i=\hat{x}_1^i\hat x_2^{j-i}\hat b=\hat x_1^{i+1}\hat x_2^{j-i+1}(1-\hat x_1-\hat x_2), \ i=0,1,\cdots,j
\end{equation*}
and $${\Xi}_j=\text{span} \{\hat{\psi}_i\mathbb{T}_s: \ i=0,1,\cdots,j; \ s=1,2,3\}.$$

\begin{lem}\label{lem4-5}
	For any $j\geq 1$,      $\tau\in{\Xi}_j$  satisfies   Assumption \ref{assum4.1}.
\end{lem}
\begin{proof} 
	We first show that the functions $\left\{\frac{\partial \hat{\psi_i}}{\partial \hat x_1},\ \frac{\partial \hat{\psi_i}}{\partial \hat x_2}\right\}_{i=0}^j$ are linear independent. It is easy to obtain
	\[\begin{array}{l}
	\frac{\partial \hat{\psi_i}}{\partial \hat x_1} = (i + 1)\hat x_1^i\hat x_2^{j - i + 1} - (i + 2)\hat x_1^{i + 1}\hat x_2^{j - i + 1} - (i + 1)\hat x_1^i\hat x_2^{j - i + 2},\\
	\frac{\partial \hat{\psi_i}}{\partial \hat x_2}= (j - i + 1)\hat x_1^{i + 1}\hat x_2^{j - i} - (j - i + 1)\hat x_1^{i + 2}\hat x_2^{j - i} - (j - i + 2)\hat x_1^{i + 1}\hat x_2^{j - i + 1}.
	\end{array}\]
	Suppose that there are constants $\{c_i\}_{i=0}^j,\ \{d_i\}_{i=0}^j$ such that
	\begin{equation*}
	\sum\limits_{i = 0}^j {{c_i}\frac{\partial \hat{\psi_i}}{\partial \hat x_1}}  + \sum\limits_{i = 0}^j {{d_i}{\frac{\partial \hat{\psi_i}}{\partial \hat x_2}}}  = 0,
	\end{equation*}
	which indicates, 
	for $0\leq i\leq j +1$,
	\begin{alignat}{1}
	&(i + 1){c_i} + (j - i +2 ){d_{i - 1}} = 0,\nonumber\\
	&(i + 1){c_i} + (i + 1){c_{i - 1}} + (j - i + 3){d_{i - 2}} + (j - i + 3){d_{i-1}} = 0.\nonumber
	\end{alignat}
	Here   we set $c_{-1}=d_{-1}=d_{-2}=c_{j+1}=0$. 
	Simple calculations show 
	%
	that  	$$c_i=d_i=0,\quad   i=0,1,\cdots j, $$ i.e.     $\left\{\frac{\partial \hat{\psi_i}}{\partial \hat x_1},\ \frac{\partial \hat{\psi_i}}{\partial \hat x_2}\right\}_{i=0}^j$ are linear independent.

	Second, for any $\hat{\tau}\in \hat{\Xi}_j$, there exist constants $c_{is}(i=0,1,\cdots,j;\ s=1,2,3)$, such that
	$$\hat{\tau}=\sum_{i=0}^{j}\sum_{s=1}^{3}c_{is}\hat{\psi_i}\mathbb{T}_s,$$
	which  means
	$$\widehat{\rm div}\hat{\tau}=	\left(\begin{array}{l}
	\sum\limits_{i=0}^{j}\left( c_{i1}\frac{\partial \hat{\psi_i}}{\partial \hat x_1} + c_{i3}\frac{\partial \hat{\psi_i}}{\partial \hat x_2}\right)\\
	\sum\limits_{i=0}^{j}\left( c_{i2}\frac{\partial \hat{\psi_i}}{\partial \hat x_2} + c_{i3}\frac{\partial \hat{\psi_i}}{\partial \hat x_1}\right)\\
	\end{array}\right).$$
	
	If $\widehat{\rm div}\hat{\tau}  =0$, then we get 
	$$\sum\limits_{i=0}^{j}\left( c_{i1}\frac{\partial \hat{\psi_i}}{\partial \hat x_1} + c_{i3}\frac{\partial \hat{\psi_i}}{\partial \hat x_2}\right)=0, \quad \sum\limits_{i=0}^{j}\left( c_{i2}\frac{\partial \hat{\psi_i}}{\partial \hat x_2} + c_{i3}\frac{\partial \hat{\psi_i}}{\partial \hat x_1}\right)=0.$$
 Thus,   from the  linear independence of $\left\{\frac{\partial \hat{\psi_i}}{\partial \hat x_1},\ \frac{\partial \hat{\psi_i}}{\partial \hat x_2}\right\}_{i=0}^j$
it follows
	$$c_{is}=0, \ \ i=0,1,\cdots, j, \ s=1,2,3,$$
i.e. 	  $\hat{\tau}=0$. This completes the proof.
\end{proof}

Thanks to Lemma \ref{lem4-5}, we can obtain the following estimate for  the consistency error.

\begin{lem}\label{lem4-6}
	For any $\tilde{\tau}_h\in{\Sigma}_{k,k',h}$, it holds
		\begin{equation}
		{\mathop {\sup }\limits_{{\tau_h} \in {{\Sigma} _{k,k',h}}} \frac{{\left| {{E_h}(\tilde{\tau}_h,{\tau _h})} \right|}}{{{{\left\| {{\tau _h}} \right\|}_{H({\rm{div}})}}}}}\lesssim h^k\left|\tilde{\tau}_h\right|_{k,h}.
		\end{equation}
\end{lem}
\begin{proof} In view of the definition, \eqref{3-5.06}, of $\Sigma _{k,k',h}$,
	we have for any $\tilde \tau_h\in \Sigma _{k,k',h}$, 
	\[\mathop {\sup }\limits_{\tau_h  \in \Sigma _{k,k',h} } \frac{{{E_h}(\tilde \tau_h ,\tau_h )}}{{{{\left\| \tau_h  \right\|}_{H({\rm div})}}}} \le \mathop {\sup }\limits_{\tau_h  \in \Sigma _{k,k',h}\backslash\Xi_{k'-3}} \frac{{{E_h}(\tilde \tau_h ,\tau_h )}}{{{{\left\| \tau_h  \right\|}_{H({\rm div})}}}} + \mathop {\sup }\limits_{\tau_h  \in \Xi_{k'-3}} \frac{{{E_h}(\tilde \tau_h ,\tau_h )}}{{{{\left\| \tau_h  \right\|}_{H({\rm div})}}}}=: {M_1} + {M_2}.\]
	
	We first estimate $M_1$. Since the degree of polynomials contained in $\Sigma _{k,k',h}\backslash\Xi_{k'-3}$ is at most $k'-1$,   we can take $\tilde{k}=k'-1, p=k, q=0$ in Lemma \ref{lem4-2} to get
	\begin{equation}\label{4-36}
 {{M_1}}  = \mathop {\sup }\limits_{\tau_h  \in \Sigma _{k,k',h}\backslash\Xi_{k'-3}} \frac{{{E_h}(\tilde \tau_h ,\tau_h )}}{{{{\left\| \tau _h \right\|}_{H({\rm div})}}}} \lesssim 
 {h^k}{\left| \tilde \tau_h  \right|_{k,h}}.
	\end{equation}
	
	For $M_2$,  take
	$\tilde{k}=k',\ p=k-1, q=1$ in Lemma \ref{lem4-2},
	then by Lemma \ref{lem3.3} and Lemma \ref{lem4-5} we obtain
	\begin{eqnarray*}
 {{M_2}} &=& \mathop {\sup }\limits_{\tau_h  \in {\Xi _{k' - 3}}} \frac{{{E_h}({\tilde \tau_h},\tau_h )}}{{{{\left\| \tau_h  \right\|}_{H({\rm{div}})}}}} \lesssim \mathop {\sup }\limits_{\tau_h  \in \Xi _{k' - 3}} \frac{{{h^k}{{\left| {{\tilde \tau_h}} \right|}_{k - 1,h}}{{\left| \tau_h  \right|}_{1,h}}}}{{{{\left\|{\rm div} \tau_h  \right\|}_0}}}
 \lesssim \mathop {\sup }\limits_{\tau_h  \in \Xi _{k' - 3}} \frac{{{h^k}{{\left| {{\tilde \tau_h}} \right|}_{k - 1,h}}{{\left| \tau_h  \right|}_{1,h}}}}{{{{\left| \tau_h  \right|}_{1,h}}}} \\
& \lesssim &{h^k}{\left| {{\tilde \tau_h}} \right|_{k - 1,h}},
	\end{eqnarray*}
which, together with   (\ref{4-36}), yields the desired conclusion.
\end{proof}

Finally, combining  Lemma \ref{lem4-6} and Lemma \ref{lem4.3} immediately yields  the following optimal error estimate for the mass lumping mixed finite element scheme.

\begin{thm}\label{thm4.1}	
	Let $(\sigma,u)\in \left(\Sigma\bigcap  H^{k+1}(\Omega;\mathbb{S})\right)\times \left(V\bigcap  H^k(\Omega;\mathbb{R}^2)\right)$ and $(\sigma_h,u_h)\in{\Sigma}_{k,k',h}\times V_{k,k',h}$ be the solutions of  (\ref{continuous}) and (\ref{discrete2}), respectively. Then 
	\begin{equation}\label{4-35}
	\begin{gathered}
	\left\|\sigma-\sigma_h\right\|_{H({\rm div})} + \left\|u-u_h\right\|_0\lesssim h^{k}\left(\left\|\sigma\right\|_{k+1} +\left\|u\right\|_{k}\right).
	\end{gathered}
	\end{equation}
\end{thm}

The following theorem shows that the optimal error estimate for the  stress in $L^2$ norm can be achieved for some special cases.
\begin{thm}
	Let $(\sigma_h,v_h)\in{\Sigma}_{k,k',h}\times V_{k,k',h}$ be the solution of (\ref{discrete2}) with $k'\geq k+2$. If the tensor functions contained in the space ${\Xi}_{k'-3}\bigcup {\Xi}_{k'-4}$ satisfy   Assumption \ref{assum4.1}, then 
	\begin{equation}\label{4-37}
	{\left\| {\sigma  - {\sigma _h}} \right\|_0} \lesssim{h^{k + 1}}{\left\| \sigma  \right\|_{k + 1}}.
	\end{equation}
\end{thm}
\begin{proof}	Let $\Pi_h$ be   the same operator as in \eqref{commu}, then 
	  it suffices to show
	\[{\left\| {{\Pi _h}\sigma  - {\sigma _h}} \right\|_0} \lesssim {h^{k + 1}}{\left\| \sigma  \right\|_{k + 1}}.\]
	In fact,  we can write 
	 $$
	 \Pi_h\sigma-\sigma_h=\sigma_1+\sigma_2$$ with $\sigma_1\in\Sigma_{k,k',h}\setminus\left({{\Xi}_{k'-3}\cup{\Xi}_{k'-4}}\right)$ and $\sigma_2\in{\Xi}_{k'-3}\cup{\Xi}_{k'-4}$. By the community property (\ref{commu}), we get
	$$ {\rm div}(\sigma_1+\sigma_2) ={\rm div}(\Pi_h\sigma-\sigma_h)=0.$$
	Since the degree of the polynomial vector ${\rm div}\sigma_1$ is   $k'-3$, then the degree of ${\rm div}\sigma_2$ is also no more than $k'-3$. According to  Assumption \ref{assum4.1},  we can derive that 
	 $\sigma_2\in {\Xi}_{k'-4}=b  P^{hom}_{k'-4}(K; \mathbb{S}).$  This implies that    ${\rm div}\sigma_2$ is of degree $k'-2$ if $\sigma_2\neq 0$, which conflicts the conclusion that ${\rm div}\sigma_2$ is   no more than $k'-3$. Hence,  $\sigma_2=0$. As a result,  $\Pi_h{\sigma}-\sigma_h=\sigma_1$ is of degree at most $k'-2$.
	%
	
	 Now we set $\tilde{k}=k'-2, p=k+1, q=0$ in Lemma \ref{lem4-2}, then
	\begin{eqnarray*}
	{E_h}({\Pi_h}\sigma ,{\Pi _h}\sigma  - {\sigma _h}) 
	&\lesssim& {h^{k + 1}}{\left| {{\Pi _h}\sigma } \right|_{k + 1,h}}{\left\| {{\Pi _h}\sigma  - {\sigma _h}} \right\|_0} \\
	&\lesssim & h{\left| \sigma  \right|_{k + 1}}{\left\| {{\Pi _h}\sigma  - {\sigma _h}} \right\|_0}.
	\end{eqnarray*}
	From (\ref{continuous}),  (\ref{discrete2}) and Lemma \ref{lem4.2}, it follows
	\begin{equation*}\label{4-9}
	\begin{array}{rl}
	\left\| {\Pi_h\sigma  - {\sigma _h}} \right\|_0^2 &\lesssim{(\mathcal{A}(\Pi_h\sigma - {\sigma _h}), \Pi_h\sigma - {\sigma _h})_h}\\
	&=- (\mathcal{A}(\sigma - \Pi_h\sigma ),\Pi_h\sigma - {\sigma_h}) - {E_h}(\Pi_h\sigma, \Pi_h\sigma - {\sigma_h}).
	\end{array}
	\end{equation*}
	Combining the two estimates above indicates
	\[{\left\| {{\Pi _h}\sigma  - {\sigma _h}} \right\|_0} \lesssim{\left\| {\sigma  - {\Pi _h}\sigma } \right\|_0} + \frac{{{E_h}({\Pi _h}\sigma ,{\Pi _h}\sigma  - {\sigma _h})}}{{{{\left\| {{\Pi _h}\sigma  - {\sigma _h}} \right\|}_0}}} \lesssim{h^{k + 1}}{\left\| \sigma  \right\|_{k + 1}}.\]
	This finishes the proof.  
\end{proof}

\begin{rem}\label{rem4.2}
	We can verify that the space ${\Xi}_1\cup{\Xi}_2$ satisfies Assumtion \ref{assum4.1}. Thus,  for  the solution  $(\sigma_h,v_h)\in{\Sigma}_{3,5,h}\times V_{3,5,h}$, of (\ref{discrete2}),  it holds  
	\begin{equation}
	{\left\| {\sigma  - {\sigma _h}} \right\|_0} \lesssim{h^{4}}{\left\| \sigma  \right\|_{4}}.
	\end{equation}
\end{rem}

\section{Numerical results}

In this section, we  shall  give a numerical example to verify our theoretical analysis for the  scheme  \eqref{discrete1}, of the modified  mixed element $\Sigma_{k,k',h}-V_{k,k',h}$,   and the mass lumping scheme \eqref{discrete2} in three cases:   $k=3,k'=4$; $k=4,k'=5$; $k=3,k'=5$.

Take   $\Omega=[0,1]\times[0,1]$ and Lam$\acute{e}$ constants $\lambda=1, \mu=\frac{1}{2}$ in the model problem  (\ref{continuous}). Let the exact solution $(\sigma,u)$ be of the following form: 
\begin{eqnarray*}
u_1 &=  &- x_1^2{x_2}(2{x_2} - 1){({x_1} - 1)^2}({x_2} - 1),\\
u_2 &=& {x_1}x_2^2(2{x_1} - 1){({x_2} - 1)^2}({x_1} - 1),\\
\sigma_{11} &=&-\sigma_{22}= -2x_1x_2(2x_1^2 - 3x_1 +1)(2x_2^2-3x_2+1),\\
\sigma_{12} &=&\sigma_{21}= x_1x_2^2(x_2-1)^2(2x_1-\frac{3}{2}) - x_1^2x_2(x_1-1)^2(2x_2-\frac{3}{2})\\
 &&\quad \quad \qquad -\frac{x_1^2}{2}(2x_2-1)(x_1-1)^2(x_2-1)+\frac{x_2^2}{2}(2x_1-1)(x_1-1)(x_2-1)^2.
\end{eqnarray*}
We use $N\times N$ uniform triangular meshes for the computation  (cf. Figure \ref{fig1}), and list the    error results of the  stress and displacement approximations in Tables 2-4. 

	 Table 2 gives the  results of Hu-Zhang's element $\Sigma_{k,h}-V_{k,h}$  \cite{Hu2015Family}, the modified element  $\Sigma_{k,k',h}-V_{k,k',h}$ and the mass lumping scheme   for $k=3, k'=4$. 
	 Table 3   gives the  results of the three methods for $k=4, k'=5$.  And Table 4 gives the results of  the modified element  $\Sigma_{k,k',h}-V_{k,k',h}$ and the mass lumping scheme  for $k=3, k'=5$.    
		From the numerical results we have the following observations:
	\begin{itemize}	
		\item  As  same as Hu-Zhang's element, the modified element $\Sigma_{k,k',h}-V_{k,k',h}$ for $k=3,4$ yields the $k$-th order of convergence for  $\left\|{\rm div}(\sigma-\sigma_h)\right\|_0$ and $\left\|u-u_h\right\|_0$, and $k+1$-th order of convergence for $\left\|\sigma-\sigma_h\right\|_0$. This is conformable to the theoretical results in 
Theorem \ref{thm3-1}.
		
\item The mass lumping scheme of the modified element $\Sigma_{k,k',h}-V_{k,k',h}$  yields the $k$-th order of convergence for  $\left\|{\rm div}(\sigma-\sigma_h)\right\|_0$ and $\left\|u-u_h\right\|_0$, 
as is conformable to the theoretical result in 
Theorem \ref{thm4.1}.

\item The mass lumping scheme  of   $\Sigma_{k,k',h}-V_{k,k',h}$, with  $k=3, k'=4$ and $k=4, k'=5$, yields the $k$-th order of convergence for   $\left\|\sigma-\sigma_h\right\|_0$, one order lower than the original scheme, while the mass lumping scheme  with $k=3, k'=5$ yields the $k+1$-th order of convergence,   which is consistent with Remark \ref{rem4.2}.

\item Though the proposed modified element $\Sigma_{k,k',h}-V_{k,k',h}$ is of more degrees of freedom than Hu-Zhang's element $\Sigma_{k,h}-V_{k,h}$, its  mass lumping scheme   leads to a SPD   system that is much easier to solve. 
		\end{itemize}

\begin{figure}[htbp]
	\centering
	\subfigure[$2\times 2$]{
		\begin{minipage}[t]{0.35\linewidth}
			\centering
			\includegraphics[width=4cm]{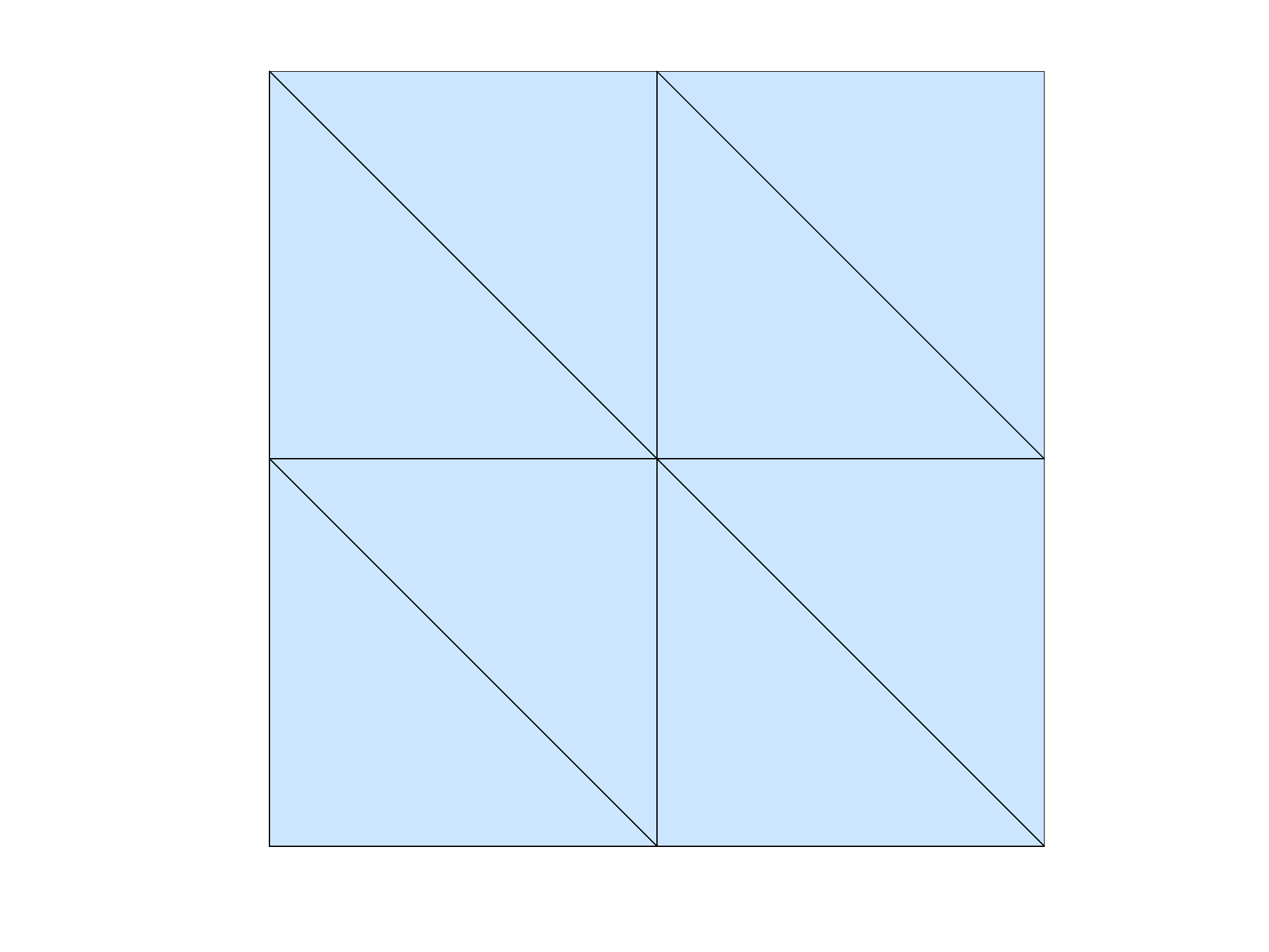}
		\end{minipage}
	}%
	\subfigure[$4\times 4$]{
		\begin{minipage}[t]{0.35\linewidth}
			\centering
			\includegraphics[width=4cm]{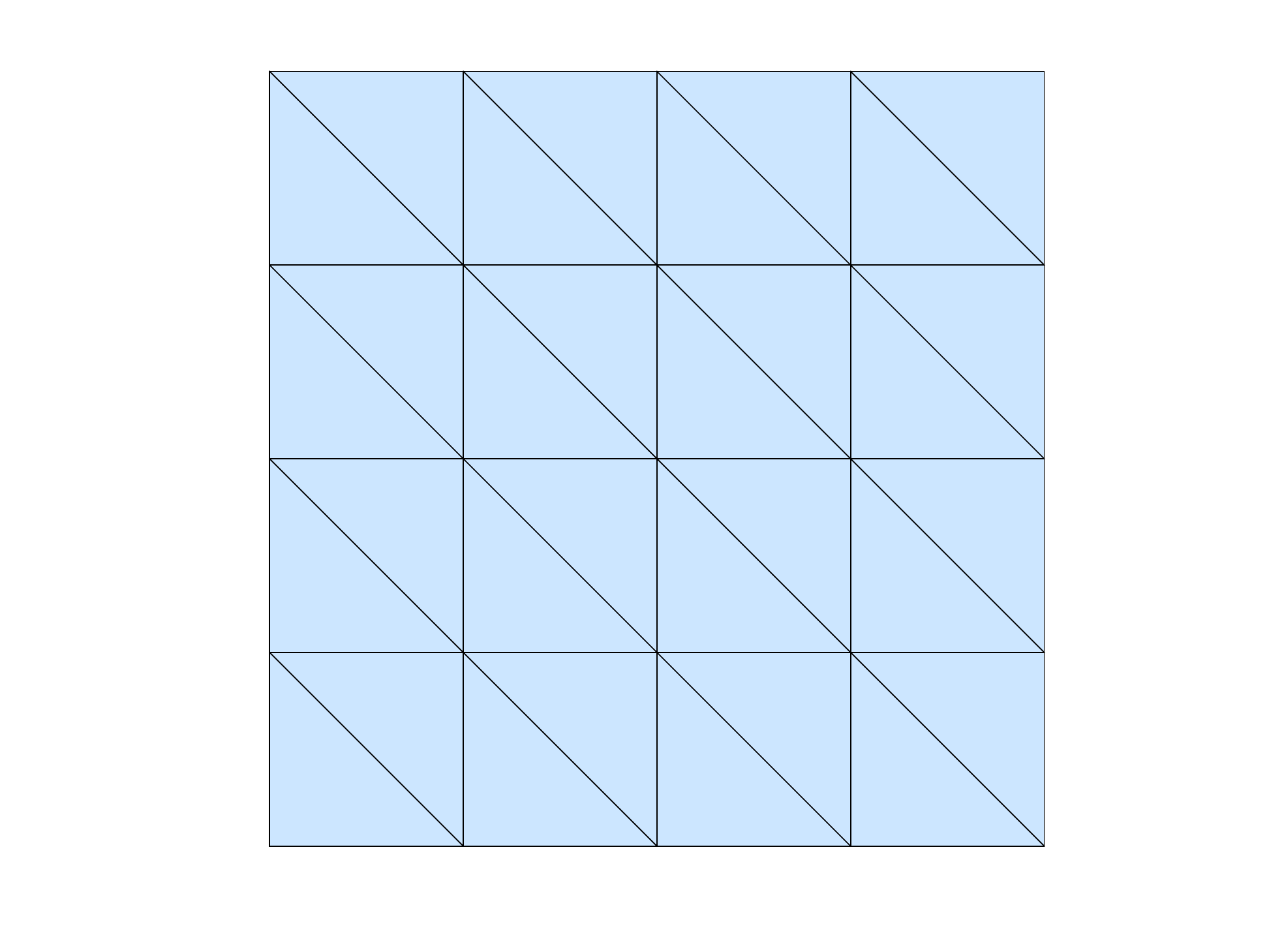}
		\end{minipage}
	}%
	\centering
	\caption{ The domain with uniform triangular meshes}\label{fig1}
\end{figure}
			
	\begin{table}
		\centering
		\caption{History of convergence: $k=3, k'=4$}
		\begin{tabular}{@{} c|c|cc|cc|cc @{}}
					\toprule
					 &\multirow{2}{*}{$N$}&
							\multicolumn{2}{c|}{$\frac{{{{\left\| {\sigma - \sigma _h} \right\|}_0}}}{{{{\left\| {\sigma} \right\|}_0}}}$} &
							\multicolumn{2}{c|}{$\frac{{\left\| {\rm div}({\sigma  - {\sigma _h}}) \right\|}_0}{{\left\| {\rm div}\sigma  \right\|}_0}$} &
							\multicolumn{2}{c}{$\frac{{{{\left\| {u - u_h} \right\|}_0}}}{{{{\left\| {u} \right\|}_0}}}$}\\
			 &&Error&Order&Error&Order&Error&Order\\
			\hline
			&$ 2$&9.361e-2&-- &  9.256e-2& --   &1.409e-1&--\\
			Hu-Zhang's element&$4$& 9.035e-3&3.37&1.480e-2&2.64&1.948e-2&2.85\\
			$\Sigma_{3,h}-V_{3,h}$&$8$& 6.498e-4&3.79&1.953e-3&2.92&2.590e-3&2.91\\
			&$16$&4.289e-5&3.92&2.473e-4&2.98&3.296e-4&2.97\\
			&$32$&2.742e-6&3.96&3.102e-5&2.99&4.139e-5&2.99\\
			\hline
			&$2$&1.065e-1&-- &  5.414e-2& --   &7.038e-2&--\\
			Modified element &$4$& 1.120e-2&3.25&7.438e-3&2.86&9.685e-3&2.86\\
			$\Sigma_{3,4,h}-V_{3,4,h}$ &$8$& 8.296e-4& 3.75&9.496e-4&2.96&1.240e-3&2.96\\
			&$16$&5.551e-5& 3.90&1.193e-4&2.99&1.565e-4&2.98\\
			&$32$&3.573e-6&3.95&1.493e-5&3.00&1.962e-5& 2.99\\
			\hline
			&$2$&1.219e-1&-- &  6.417e-2& --   &8.983e-2&--\\
			Mass lumping &$4$& 1.731e-2&2.81&7.880e-3&3.02&1.327e-2&2.75\\
			$\Sigma_{3,4,h}-V_{3,4,h}$&$8$& 2.0759e-3& 3.06&9.741e-4&3.01&1.758e-3&2.91\\
			&$16$&2.466e-4& 3.07&1.213e-4&3.00&2.232e-4&2.97\\
			&$32$&2.981e-5&3.04&1.515e-5&3.00&2.801e-5& 2.99\\			\hline
		\end{tabular}
	\end{table}
	
		\begin{table}
			\centering
			\caption{History of convergence: $k=4, k'=5$}
			\begin{tabular}{@{} c|c|cc|cc|cc @{}}
				\toprule
				&\multirow{2}{*}{$N  $}&
				\multicolumn{2}{c|}{$\frac{{{{\left\| {\sigma - \sigma _h} \right\|}_0}}}{{{{\left\| {\sigma} \right\|}_0}}}$} &
				\multicolumn{2}{c|}{$\frac{{\left\| {\rm div}({\sigma  - {\sigma _h}}) \right\|}_0}{{\left\| {\rm div}\sigma  \right\|}_0}$} &
				\multicolumn{2}{c}{$\frac{{{{\left\| {u - u_h} \right\|}_0}}}{{{{\left\| {u} \right\|}_0}}}$}\\
				&&Error&Order&Error&Order&Error&Order\\
				\hline
				&$2$&1.919e-2&-- &2.505e-2& --   &2.583e-2&--\\
			Hu-Zhang's element	&$4$&7.329e-4&4.71&1.724e-3&3.86&2.655e-3&3.28\\
				$\Sigma_{4,h}-V_{4,h}$&$8$&2.481e-5&4.88&1.101e-4&3.96&1.860e-4&3.83\\
				&$16$&8.043e-7&4.94&6.919e-6&3.99&1.194e-5&3.96\\
				&$32$&2.557e-8&4.97&4.330e-7&4.00&7.519e-7&3.99\\
				\hline
				&$2$&2.602e-2&-- &4.862e-3& --   &1.403e-2&--\\
				Modified element &$4$&9.792e-4&4.73&2.239e-4&4.44&6.087e-4&4.52\\
				$\Sigma_{4,5,h}-V_{4,5,h}$&$8$&3.302e-5& 4.88&1.243e-5&4.17&3.298e-5&4.20\\
				&$16$&1.069e-6&4.94&7.508e-7&4.04&1.980e-6&4.05\\
				&$32$&3.401e-8&4.97&4.650e-8&4.01&1.225e-7&4.01\\
				\hline		
				&$2$&3.679e-2&-- &6.097e-3& --   &1.751e-2&--\\
				Mass lumping &$4$&2.377e-3&3.95&2.532e-4&4.58&1.753e-3&3.32\\
				$\Sigma_{4,5,h}-V_{4,5,h}$ &$8$&1.499e-4&3.98&1.308e-5&4.27&1.223e-4&3.84\\
				&$16$&9.369e-6&4.00&7.690e-7&4.08&7.853e-6&3.96\\
				&$32$&5.843e-7&4.00&4.727e-8&4.02&4.942e-7&3.99\\				\hline
			\end{tabular}
		\end{table}
		
				\begin{table}
					\centering
					\caption{History of convergence: $k=3, k'=5$}
					\begin{tabular}{@{} c|c|cc|cc|cc @{}}
						\toprule
						&\multirow{2}{*}{$N $}&
						\multicolumn{2}{c|}{$\frac{{{{\left\| {\sigma - \sigma _h} \right\|}_0}}}{{{{\left\| {\sigma} \right\|}_0}}}$} &
						\multicolumn{2}{c|}{$\frac{{\left\| {\rm div}({\sigma  - {\sigma _h}}) \right\|}_0}{{\left\|{\rm div} \sigma  \right\|}_0}$} &
						\multicolumn{2}{c}{$\frac{{{{\left\| {u - u_h} \right\|}_0}}}{{{{\left\| {u} \right\|}_0}}}$}\\
						&&Error&Order&Error&Order&Error&Order\\
						\hline
				&	$2$&1.140e-1&-- &  3.226e-2& --   &5.201e-2&--\\
				 Modified element	&$4$&1.185e-2&3.26&4.176e-3&2.95&5.757e-3&3.17\\
					$\Sigma_{3,5,h}-V_{3,5,h}$	&$8$&8.745e-4& 3.76&5.248e-4&2.99&6.694e-4&3.10\\
						&$16$&5.841e-5& 3.90&6.567e-5&3.00&8.354e-5&3.00\\
						&$32$&3.757e-6&3.95&8.211e-6&3.00&1.045e-5& 3.00\\
						\hline
						&$2$&1.131e-1&-- &  3.575e-2& --   &6.825e-2&--\\
					Mass lumping 	&$4$&1.184e-2&3.25&4.621e-3&2.95&6.751e-3&3.33\\
					$\Sigma_{3,5,h}-V_{3,5,h}$	&$8$&8.751e-4& 3.75&5.809e-4&2.99&7.929e-4&3.08\\
						&$16$&5.850e-5& 3.90&7.271e-5&3.00&9.862e-5&3.00\\
						&$32$&3.764e-6&3.95&9.091e-6&3.00&1.233e-5& 3.00\\
						\hline						
					\end{tabular}
				\end{table}

	\bibliographystyle{plain}
	\bibliography{elasticity_ref_papers}
\end{document}